\documentclass[a4paper]{amsart}

\usepackage{epic, eepic, amsfonts, latexsym, amssymb, graphicx,
multicol, mathrsfs, color, amscd, verbatim, paralist,
xspace, url, euscript, stmaryrd,  amsmath, enumitem,
bbold, multirow, tikz, mathtools}
\usepackage[all,pdf,cmtip]{xy}

\usepackage[colorlinks, linkcolor=blue, citecolor=magenta, urlcolor=cyan]{hyperref}

\def\tilde{\widetilde}
\def\hat{\widehat}
\renewcommand\bar{\overline}

\def\fm{{\mathfrak m}}

\def\CC{{\mathbb C}}

\def\PP{{\mathbb P}}

\def\hat{\widehat}

\def\im{\mathop{\rm im}\nolimits}

\def\GL{\mathop{\rm GL}\nolimits}
\def\SL{\mathop{\rm SL}\nolimits}

\def\grad{\mathop{\rm grad}\nolimits}

\def\Res{\mathop{\rm Res}\nolimits}
\def\Gor{\mathop{\rm Gor}\nolimits}

\def\Soc{\mathop{\rm Soc}\nolimits}
\def\rank{\mathop{\rm rank}\nolimits}
\def\Jac{\mathop{\rm Jac}\nolimits}
\def\jac{\mathop{\rm jac}\nolimits}

\def\Gr{\mathop{\rm Gr}\nolimits}

\newcommand{\bA}{\mathbf{A}}

\newcommand{\bbf}{\mathbf{f}}

\def\Char{\mathop{\rm char}\nolimits}

\newcommand\co{\colon} 

\newtheorem{theorem}{THEOREM}[section]

\newtheorem{proposition}[theorem]{Proposition}
\newtheorem{conjecture}[theorem]{Conjecture}

\theoremstyle{definition}

\newtheorem{prop}[theorem]{Proposition}

\newtheorem{lemma}[theorem]{Lemma}

\theoremstyle{remark}
\newtheorem{remark}[theorem]{Remark}

\makeatletter
\def\blfootnote{\xdef\@thefnmark{}\@footnotetext}
\makeatother

\begin{document}

\title[On the image of the associated form morphism]{On the image of the associated form morphism}\blfootnote{{\bf Mathematics Subject Classification:} 14L24, 13A50, 32S25}\blfootnote{{\bf Keywords:} classical invariant theory, geometric invariant theory, catalecticant varieties, associated forms, isolated hypersurface singularities.}
\author[Isaev]{Alexander Isaev}

\address{Mathematical Sciences Institute\\
Australian National University\\
Acton, ACT 2601, Australia}
\email{alexander.isaev@anu.edu.au}

\maketitle

\thispagestyle{empty}

\pagestyle{myheadings}

\begin{abstract} 
Let $\CC[x_1,\dots,x_n]_{d+1}$ be the vector space of homogeneous forms of degree $d+1$ on $\CC^n$, with $n,d\ge 2$. In earlier articles by J. Alper, M. Eastwood and the author, we introduced a morphism, called $A$, that assigns to every nondegenerate form the so-called associated form lying in the space $\CC[y_1,\dots,y_n]_{n(d-1)}$. One of the reasons for our interest in $A$ is the conjecture---motivated by the well-known Mather-Yau theorem on complex isolated hypersurface singularities---asserting that all regular $\GL_n$-invariant functions on the affine open subvariety $\CC[x_1,\dots,x_n]_{d+1,\Delta}$ of forms with nonvanishing discriminant can be obtained as the pull-backs by means of $A$ of the rational $\GL_n$-invariant functions on $\CC[y_1,\dots,y_n]_{n(d-1)}$ defined on $\im(A)$. The morphism $A$ factors as $A=\bA\circ \grad$, where $\grad$ is the gradient morphism and $\bA$ assigns to every $n$-tuple of forms of degree $d$ with nonvanishing resultant a form in $\CC[y_1,\dots,y_n]_{n(d-1)}$ defined analogously to $A(f)$ for a nondegenerate $f$. In order to establish the conjecture, it is important to study the image of $\bA$. In the present paper, we show that $\im(\bA)$ is an open subset of an irreducible component of each of the so-called catalecticant varieties $V$, $\Gor(T)$ and describe the closed complement to $\im(\bA)$, at the same time clarifying and extending known results on these varieties. Furthermore, for $n=3$, $d=2$ we give a description of the complement to $\im(\bA)$ via the zero locus of the Aronhold invariant of degree 4, which establishes an analogy with the case $n=2$ where this complement is known to be the vanishing locus of the catalecticant for any $d\ge 2$.
\end{abstract}

\section{Introduction}\label{intro}
\setcounter{equation}{0}

This paper is motivated by a new construction in classical invariant theory that originated in article \cite{EI} and was further explored in \cite{AI1}, \cite{AI2}, \cite{AIK}. Fix integers $n\ge 2$ and $d\ge 2$ and let $\CC[x_1,\dots,x_n]_{d+1,\Delta}$ be the complex affine open subvariety of the space $\CC[x_1,\dots,x_n]_{d+1}$ of homogeneous forms of degree $d+1$ in $n$ variables where the discriminant $\Delta$ does not vanish. Consider the Milnor algebra $M(f) := \CC[x_1,\dots,x_n]/(f_{x_{{}_1}}, \ldots, f_{x_{{}_n}})$ of the isolated singularity at the origin of the hypersurface in $\CC^n$ defined by $f\in  \CC[x_1,\dots,x_n]_{d+1,\Delta}$ and let $\fm \subset M(f)$ be the maximal ideal. One can then introduce a form on the $n$-dimensional quotient $\fm/\fm^2$ with values in the one-dimensional socle $\Soc(M(f))$ of $M(f)$ as follows:
$$
\fm/\fm^2 	 \to \Soc(M(f)), \quad z  \mapsto \hat z^{\,n(d-1)},
$$
where $\hat z$ is any element of ${\mathfrak m}$ that projects to $z\in{\mathfrak m}/{\mathfrak m}^2$.  There is a canonical isomorphism ${\mathfrak m}/{\mathfrak m}^2\cong \CC^n$ and, since the Hessian of $f$ generates the socle, there is also a canonical isomorphism $\Soc(M(f)) \cong \CC$. Hence, one obtains a form $A(f)$ of degree $n(d-1)$ on  $\CC^n$, which is called the {\it associated form of}\, $f$. This form is very natural; in particular, it is a Macaulay inverse system for the Milnor algebra $M(f)$.

The main object of our study in \cite{AI1}, \cite{AI2}, \cite{AIK} was the morphism
$$
A\co \CC[x_1,\dots,x_n]_{d+1,\Delta} \to \CC[y_1,\dots,y_n]_{n(d-1)}, \quad f \mapsto A(f)
$$
of affine varieties. As first observed in \cite{EI}, for certain values of $n$ and $d$ one can recover all $\GL_n$-invariant rational functions on forms of degree $d+1$ from those on forms of degree $n(d-1)$ by evaluating the latter on associated forms, i.e., by composing them with $A$. Motivated by the above fact, in \cite{AI1} we proposed a conjecture asserting that an analogous statement holds for all $n$ and $d$ (cf. \cite[Conjecture 3.2]{EI}):

\begin{conjecture}\label{C:main} For any regular $\GL_n$-invariant function $J$ on $\CC[x_1,\dots,x_n]_{d+1,\Delta}$ there exists a rational\, $\GL_n$-invariant function $\tilde J$ on $\CC[y_1,\dots,y_n]_{n(d-1)}$ defined on the image of $A$ such that $J=\tilde J \circ A$.
\end{conjecture} 

In other words, the conjecture asserts that the invariant theory of forms of degree $d+1$ can be extracted, in a canonical way, from that of forms of degree $n(d-1)$ at least at the level of rational invariant functions. While this statement is quite intriguing from the purely invariant-theoretic viewpoint, it was originally motivated---as explained in \cite{EI}, \cite{AI1}---by complex singularity theory, specifically, by the well-known Mather-Yau theorem (see \cite{MY} and also \cite{Be}, \cite{Sh}, \cite[Theorem 2.26]{GLS}). In \cite{AI2}, Conjecture \ref{C:main} was shown to hold for binary forms of any degree, and in \cite{AI1} its weaker variant was established for arbitrary $n, d$.

In this paper, we obtain results towards settling the conjecture in full generality, which are at the same time of interest in a broader algebraic context. The morphism $A$ factors as $A=\bA\circ \grad$, where $\grad:\CC[x_1,\dots,x_n]_{d+1} \to\CC[x_1,\dots,x_n]_{d}^{\oplus n}$
is the gradient morphism and $\bA:(\CC[x_1,\dots,x_n]_{d}^{\oplus n})_{\Res}\to\CC[y_1,\dots,y_n]_{n(d-1)}$ assigns to every $n$-tuple $\bbf=(f_1,\dots,f_n)$ of forms of degree $d$ with nonvanishing resultant the {\it associated form of}\, $\bbf$ defined analogously to $A(f)$, with the partial derivative $f_{x_{{}_j}}$ replaced by $f_j$ for all $j$. Note that for every $\bbf$ the form $\bA(\bbf)$ is a Macaulay inverse system for the zero-dimensional complete intersection algebra $M(\bbf) := \CC[x_1, \ldots, x_n] / (f_1, \ldots, f_n)$. As explained in \cite[Section 3]{AI2}, in order to establish Conjecture \ref{C:main} for all $n,d$, it is important to study the image of $\bA$. In this paper we show that $\im(\bA)\subset \CC[y_1,\dots,y_n]_{n(d-1)}$ is an open subset of an irreducible component of each of the {\it catalecticant varieties}\, $V$ and $\Gor(T)$ (see Section \ref{schemes} for the definitions) and give a description of the closed complement to $\im(\bA)$. We note that a number of other properties of the morphism $\bA$ (as well as the gradient morphism) essential for confirming Conjecture \ref{C:main} were obtained in the recent paper \cite{F}.

The irreducible components of catalecticant varieties are of general interest and have been studied regardless of Conjecture \ref{C:main} (see \cite[Chapter 4]{IK} and references therein for details). In particular, in \cite[Theorem 4.17]{IK} it was shown that $\Gor(T)$ has an irreducible component containing $\im(\bA)$ as a dense subset and the dimension of this component was found. On the other hand, an analogous fact for $V$ (which is the catalecticant variety most relevant to our study of the morphism $\bA$) appears to be only known in the cases (i) $n=3$, $d\ge 3$, (ii) $n=4$, $d=2,3$, (iii) $n=5$, $d=2$ (see \cite[Theorem 4.19 and Corollary 4.18]{IK}), and one of our aims is to bring the results on $V$ in line with those on $\Gor(T)$.

In this paper, we, first of all, refine and extend Theorems 4.17 and 4.19 of \cite{IK}. Namely, in Section \ref{schemes} we show that the set $\im(\bA)$ is open (not just dense) in an irreducible component of each of $V$, $\Gor(T)$ for all $n,d$ and explicitly describe the closed complement to $\im(\bA)$ (see Theorem \ref{main8}). Note that finding a suitable characterization of this complement is important for resolving Conjecture \ref{C:main} (see Remark \ref{complement}). As the proof of Theorem 4.17 in \cite{IK} is quite brief, we also provide an alternative derivation---with full details---of the dimension formula for $\im(\bA)$. Note that, although we assume the base field to be $\CC$, our arguments work for any algebraically closed field $k$ of characteristic zero and even apply to the case $\Char(k)>n(d-1)$, with $n(d-1)$ being the socle degree of $M(\bbf)$ for all $\bbf=(f_1,\dots,f_n)\in(k[x_1,\dots,x_n]_{d}^{\oplus n})_{\Res}$. We also stress that our clarifications and extensions of results of \cite{IK} only apply in the case of zero-dimensional complete intersections with homogeneous ideal generators of equal degrees.

In fact, ideally, one would like to have a better description of the complement to $\im(\bA)$ than the one provided by Theorem \ref{main8}. Namely, it would be desirable to represent it as the intersection of the relevant irreducible component of $V$ with the zero locus of an $\SL_n$-invariant form on $\CC[y_1,\dots,y_n]_{n(d-1)}$. This is indeed possible for $n=2$, in which case the $\SL_2$-invariant in question is the catalecticant (see \cite[Proposition 4.3]{AI2}). In Section \ref{aronhold} we show that such a representation is also valid for $n=3$, $d=2$, with the corresponding $\SL_3$-invariant being the Aronhold invariant of degree 4 (see Proposition \ref{psind3}). 

{\bf Acknowledgements.} This work is supported by the Australian Research Council.

\section{Preliminaries on associated forms and the morphism $\bA$}\label{morphism}
\setcounter{equation}{0}

In this section we introduce the main object of our study. What follows is an abridged version of the exposition given in \cite[Section 2]{AI2}.

Fix $n\ge 2$ and for any nonnegative integer $j$ define $\CC[x_1, \ldots, x_n]_j$ to be the vector space of homogeneous forms of degree $j$ in $x_1,\dots,x_n$ over $\CC$. Clearly, one has $\CC[x_1,\dots,x_n]=\oplus_{j=0}^{\infty}\CC[x_1, \ldots, x_n]_j$. Next, fix $d\ge 2$ and consider the vector space $\CC[x_1, \ldots, x_n]_d^{\oplus n}$ of $n$-tuples $\bbf = (f_1, \ldots, f_n)$ of forms of degree $d$.  Recall that the resultant $\Res$ on the space $\CC[x_1, \ldots, x_n]_d^{\oplus n}$ is a form with the property that $\Res(\bbf) \neq 0$ if and only if $f_1, \ldots, f_n$ have no common zeroes away from the origin (see, e.g.,  \cite[Chapter 13]{GKZ}).

For $\bbf = (f_1, \ldots, f_n) \in \CC[x_1, \ldots, x_n]_d^{\oplus n}$, we now introduce the algebra
$$
M(\bbf) := \CC[x_1, \ldots, x_n] / (f_1, \ldots, f_n)
$$
and recall a well-known lemma (see, e.g., \cite[Lemma 2.4]{AI2} and \cite[p.~187]{SS}):

\begin{lemma}\label{fourconds} \it The following statements are equivalent:
\begin{enumerate}
	\item[\rm (1)] the resultant $\Res(\bbf)$ is nonzero;
	\item[\rm (2)] the algebra $M(\bbf)$ has finite vector space dimension;
	\item[\rm (3)] the morphism $\bbf \co {\mathbb A}^n(\CC) \to {\mathbb A}^n(\CC)$ is finite;
	\item[\rm (4)] the $n$-tuple $\bbf$ is a homogeneous system of parameters of $\CC[x_1,\dots,x_n]$, i.e., the Krull dimension of $M(\bbf)$ is $0$.
\end{enumerate}
If the above conditions are satisfied, then $M(\bbf)$ is a local standard graded complete intersection algebra whose socle $\Soc(M(\bbf))$ is generated in degree $n(d-1)$ by the image $\bar{\jac(\bbf)} \in M(\bbf)$ of the Jacobian $\jac(\bbf):= \det \Jac(\bbf)$, where $\Jac(\bbf)$ is the Jacobian matrix $\big({\partial f_i}/{\partial x_j} \big)_{i,j}$.\end{lemma}

\begin{remark}\label{hilbert}
As we pointed out in Lemma \ref{fourconds}, the algebra $M(\bbf)$ has a natural standard grading: $M(\bbf) = \bigoplus_{i=0}^{\infty} M(\bbf)_i$.  It is well-known (see, e.g., \cite[Corollary 3.3]{St}) that the corresponding Hilbert function $H(x):=\sum_{i=0}^{\infty}t_i\,x^i$, with $t_i:=\dim_\CC{M}({\mathbf f})_i$, is given by
\begin{equation}
H(x)=(x^{d-1}+\dots+x+1)^n.\label{hf}
\end{equation}
\end{remark}

Next, we let $(\CC[x_1, \ldots, x_n]_d^{\oplus n})_{\Res}$ be the affine open subvariety of $\CC[x_1, \ldots, x_n]_d^{\oplus n}$ that consists of all $n$-tuples of forms with nonzero resultant. We now define the {\it associated form}\, $\bA(\bbf) \in \CC[y_1, \ldots, y_n]_{n(d-1)}$ of $\bbf=(f_1,\dots,f_n) \in (\CC[x_1, \ldots, x_n]_d^{\oplus n})_{\Res}$ by the formula
$$(y_1 \bar{x}_1 + y_2 \bar{x}_2 + \cdots + y_n \bar{x}_n)^{n(d-1)} = \bA(\bbf)(y_1, \ldots, y_n) \cdot \bar{\jac(\bbf)} \in M(\bbf),$$
where $\bar{x}_i \in M(\bbf)$ is the image of $x_i$. It is not hard to see that the induced map
$$
\bA \co (\CC[x_1, \ldots, x_n]_d^{\oplus n})_{\Res} \to \CC[y_1, \ldots, y_n]_{n(d-1)}, \quad \bbf \mapsto \bA(\bbf)
$$
is a morphism of affine varieties. This morphism is quite natural; in particular, it has an important equivariance property (see \cite[Lemma 2.7]{AI2}). In article \cite{AI2} we studied $\bA$ in relation to Conjecture \ref{C:main} stated in the introduction. 

We will now interpret $\bA$ in different terms. Recall that the algebra $\CC[y_1, \ldots, y_n]$ is a $\CC[x_1, \ldots, x_n]$-module via differentiation:
\begin{equation}
(h \diamond F) (y_1, \ldots, y_n) := h\left(\frac{\partial}{\partial y_1}, \ldots, \frac{\partial}{\partial y_n}\right)F(y_1, \ldots, y_n),\label{pp}
\end{equation}
where $h \in \CC[x_1, \ldots, x_n]$ and $F \in \CC[y_1, \ldots, y_n]$.  For a positive integer $j$, differentiation induces a perfect pairing
$$ 
\CC[x_1, \ldots, x_n]_j \times \CC[y_1, \ldots, y_n]_j \to \CC, \quad (h, F) \mapsto h \diamond F;
$$
it is often referred to as the {\it polar pairing}.  For $F \in \CC[y_1, \ldots, y_n]_j$, we now introduce the homogenous ideal, called the {\it annihilator}\, of $F$,
$$
F^{\perp} := \{h\in \CC[x_1, \ldots, x_n]\, \mid \, h\diamond F = 0 \} \subset \CC[x_1, \ldots, x_n],
$$
which is clearly independent of scaling and thus is well-defined for $F$ in the projective space $\PP(\CC[y_1, \ldots, y_n]_j)$. It is well-known that the quotient $\CC[x_1, \ldots, x_n] / F^{\perp}$ is a standard graded local Artinian Gorenstein algebra of socle degree $j$ and the following holds (cf.~\cite[Lemma 2.12]{IK}):

\begin{proposition} \label{prop-correspondence}
The correspondence $F \mapsto \CC[x_1, \ldots, x_n]/F^{\perp}$ induces a bijection
$$
\PP(\CC[y_1, \ldots, y_n]_j)  \to
\left\{ 
 	\begin{array}{l} 
		\text{local Artinian Gorenstein algebras $\CC[x_1, \ldots, x_n]/I$}\\ 
		\text{of socle degree $j$, where the ideal $I$ is homogeneous}\\
		 \end{array} \right\}.
$$
\end{proposition}

\begin{remark}\label{invsysrem}
\noindent Given a homogenous ideal $I \subset \CC[x_1, \ldots, x_n]$ such that $\CC[x_1, \ldots, x_n]/I$ is a local Artinian Gorenstein algebra of socle degree $j$, Proposition \ref{prop-correspondence} implies that there is a  form $F \in  \CC[y_1, \ldots, y_n]_j$, unique up to scaling, such that $I = F^{\perp}$. In fact, the uniqueness part of this statement can be strengthened: if $I\subset F^{\perp}$, then $I = F^{\perp}$ and all forms with this property are mutually proportional. Indeed, $I\subset F^{\perp}$ implies $I_j\subset F^{\perp}$, where $I_j:=I\cap \CC[x_1, \ldots, x_n]_j$, and the claim follows from the fact that $I_j$ has codimension 1 in $\CC[x_1, \ldots, x_n]_j$. Any such form $F$ is called {\it a {\rm (}homogeneous{\rm )} Macaulay inverse system for $\CC[x_1, \ldots, x_n]/I$} and its image in $\PP(\CC[y_1, \ldots, y_n]_j)$ is called {\it the {\rm (}homogeneous{\rm )} Macaulay inverse system for $\CC[x_1, \ldots, x_n]/I$}.
\end{remark}

We have (see \cite[Proposition 2.11]{AI2}):

\begin{prop} \label{P:inverse-system} \it
For any $\bbf \in (\CC[x_1, \ldots, x_n]_{d}^{\oplus n})_{\Res}$, 
the form $\bA(\bbf)$ is a Macau\-lay inverse system for the algebra $M(\bbf)$.
\end{prop}

\noindent By Proposition \ref{P:inverse-system}, the morphism $\bA$ can be thought of as a map assigning to every element $\bbf \in (\CC[x_1, \ldots, x_n]_{d}^{\oplus n})_{\Res}$ a particular Macaulay inverse system for the algebra $M(\bbf)$.

We now let $U_{\Res} \subset \CC[y_1, \ldots, y_n]_{n(d-1)}$ be the locus of forms $F$ such that the subspace $F^{\perp} \cap \CC[x_1, \ldots, x_n]_d$ is $n$-dimensional and has a basis with nonvanishing resultant. It is easy to see that $U_{\Res}$ is locally closed in $\CC[y_1, \ldots, y_n]_{n(d-1)}$, hence is a variety (see, e.g., Proposition \ref{onto1} below for details). By Proposition \ref{P:inverse-system}, the image of $\bA$ is contained in $U_{\Res}$. Moreover, if $F \in U_{\Res}$, then for the ideal $I\subset \CC[x_1, \ldots, x_n]$ generated by $F^{\perp} \cap \CC[x_1, \ldots, x_n]_{d}$, we have the inclusion $I \subset F^{\perp}$. By Remark \ref{invsysrem}, the form $F$ is the inverse system for $\CC[x_1, \ldots, x_n]/I$, and therefore $F=\bA(\bbf)$ for some basis $\bbf=(f_1,\dots,f_n)$ of $F^{\perp} \cap \CC[x_1, \ldots, x_n]_{d}$. Thus, we have proved:

\begin{prop} \label{P:image} \it $\im(\bA)=U_{\Res}$. 
\end{prop}

The constructions of the morphism $\bA$ can be projectivized. Indeed, denote by $\Gr(n, \CC[x_1, \ldots, x_n]_{d})$ the Grassmannian of $n$-dimensional subspaces of the space $\CC[x_1, \ldots, x_n]_{d}$. The resultant $\Res$ on $\CC[x_1, \ldots, x_n]_{d}^{\oplus n}$ descends to a section, also denoted by $\Res$, of a power of the very ample generator of the Picard group of $\Gr(n, \CC[x_1, \ldots, x_n]_{d})$.  Let $\Gr(n, \CC[x_1, \ldots, x_n]_{d})_{\Res}$ be the affine open subvariety where $\Res$ does not vanish; it consists of all $n$-dimensional subspaces of $\CC[x_1, \ldots, x_n]_{d}$ having a basis with nonzero resultant. Consider the morphism 
$$
(\CC[x_1, \ldots, x_n]_{d}^{\oplus n})_{\Res} \to \Gr(n, \CC[x_1, \ldots, x_n]_{d})_{\Res}, \quad  \bbf = (f_1, \ldots, f_n) \mapsto \langle f_1, \ldots, f_n \rangle,
$$
where $\langle\,\cdot\,\rangle$ denotes linear span. Then, by the equivariance property (see \cite[Lemma 2.7]{AI2}), the morphism $\bA$ composed with the projection $\CC[y_1, \ldots, y_n]_{n(d-1)}\setminus\{0\}\to\PP(\CC[y_1, \ldots, y_n]_{n(d-1)})$ factors as
$$
(\CC[x_1, \ldots, x_n]_{d}^{\oplus n})_{\Res} \to \Gr(n, \CC[x_1, \ldots, x_n]_{d})_{\Res} \xrightarrow{\hat\bA} \PP(\CC[y_1, \ldots, y_n]_{n(d-1)}).
$$
By Proposition \ref{P:inverse-system}, the morphism $\hat\bA$ can be thought of as a map assigning to every subspace $W\in\Gr(n, \CC[x_1, \ldots, x_n]_{d})_{\Res}$ {\it the}\, Macaulay inverse system for the algebra $M(\bbf)$, where $\bbf=(f_1,\dots,f_n)$ is any basis of $W$.

Proposition \ref{P:image} implies

\begin{prop} \label{P:imagehat} \it $\im(\hat\bA)=\PP(U_{\Res})$, where $\PP(U_{\Res})$ is the image of $U_{\Res}$ in the projective space $\PP(\CC[y_1, \ldots, y_n]_{n(d-1)})$.
\end{prop}

\noindent It turns out that $\hat\bA:  \Gr(n, \CC[x_1, \ldots, x_n]_{d})_{\Res}\to \PP(U_{\Res})$ is in fact an isomorphism (see \cite[Proposition 2.13]{AI2}). This last result will be utilized in our considerations of the relevant catalecticant varieties in the next section.

\section{The catalecticant varieties and a description of $\im(\bA)$}\label{schemes}
\setcounter{equation}{0}

Let
$
K:=\dim_\CC \CC[x_1,\dots,x_n]_d={d+n-1 \choose n-1}.
$
Consider the quasiaffine variety
$
U:=U_{K-n}(n(d-1)-d, d;n)\subset \CC[y_1,\dots,y_n]_{n(d-1)}
$
and the affine subvariety
$
V:=V_{K-n}(n(d-1)-d, d;n)\subset \CC[y_1,\dots,y_n]_{n(d-1)}
$
as defined in \cite[p.~5]{IK}. Specifically, set
$
L:=\dim_{\CC} \CC[y_1,\dots,y_n]_{n(d-1)-d}={n(d-1)-d+n-1 \choose n-1}
$
and let $\{{\mathtt m}_1,\dots,{\mathtt m}_K\}$, $\{{\mathbf m}_1,\dots,{\mathbf m}_L\}$ be the standard monomial bases in the spaces $\CC[x_1,\dots,x_n]_d$ and $\CC[y_1,\dots,y_n]_{n(d-1)-d}$, respectively, with the monomials numbered in accordance with some orders, which we will fix from now on. For a form $F\in \CC[y_1,\dots,y_n]_{n(d-1)}$ let $F_j:={\mathtt m}_j\diamond F\in \CC[y_1,\dots,y_n]_{n(d-1)-d}$, $j=1,\dots,K$, where $\diamond$ is defined in (\ref{pp}). Expanding $F_1,\dots,F_K$ with respect to $\{{\mathbf m}_1,\dots,{\mathbf m}_L\}$, we obtain an $L\times K$-matrix $D(F)$ called the {\it catalecticant matrix}. Then the varieties $U$ and $V$ are described as
$$
\begin{array}{l}
U=\{F \in \CC[y_1,\dots,y_n]_{n(d-1)}\mid \rank D(F)=K-n\},\\
\vspace{-0.3cm}\\
V=\{F\in \CC[y_1,\dots,y_n]_{n(d-1)}\mid \rank D(F)\le K-n\}.
\end{array}
$$
Note that $U$ is a dense open subset of $V$ (see \cite[Lemma 3.5]{IK}).

Clearly, $V\subset \CC[y_1,\dots,y_n]_{n(d-1)}$ is the affine subvariety given by the condition of the vanishing of all $(K-n+1)$-minors of $D(F)$. Observe that for $n=2$ one has $K=d+1$, $L=d-1$, and therefore the matrix $D(F)$ has no $(K-1)$-minors, hence $V=\CC[y_1,y_2]_{2(d-1)}$. Similarly, for $n=3$, $d=2$, we have $K=6$, $L=3$, therefore $D(F)$ has no $(K-2)$-minors, hence $V=\CC[y_1,y_2,y_3]_3$. Notice that in all other cases one has $L\ge K$, and therefore $V$ is a proper affine subvariety of $\CC[y_1,\dots,y_n]_{n(d-1)}$ unless $n=2$ or $n=3$, $d=2$.

Next, let $T:=(t_0,t_1,\dots,t_{n(d-1)})=(1,n,\dots,n,1)$ be the Gorenstein sequence from the Hilbert function (\ref{hf}), which is symmetric about $n(d-1)/2$. Consider the quasiaffine variety $\Gor(T)$ that consists of all forms $F\in \CC[y_1,\dots,y_n]_{n(d-1)}$ such that the Hilbert function of the standard graded local Artinian Gorenstein algebra $\CC[x_1,\dots,x_n]/F^{\perp}$ is $T$. Clearly, $\Gor(T)$ is an open subset of the affine subvariety $\Gor_{\le}(T)\subset \CC[y_1,\dots,y_n]_{n(d-1)}$ consisting of all forms $F$ for which the Gorenstein sequence of $\CC[x_1,\dots,x_n]/F^{\perp}$ does not exceed $T$. Analogously to $V$, the variety $\Gor_{\le}(T)$ is defined by the vanishing of all $(t_i+1)$-minors of the corresponding matrices constructed analogously to $D(F)$, for $i=1,\dots,n(d-1)-1$. Following \cite{IK}, we call $V$ and $\Gor(T)$ the {\it catalecticant varieties}. 

\begin{remark}\label{moregeneralcatal}
We note that \cite{IK} introduces more general catalecticant varieties (and even schemes), but $V$ and $\Gor(T)$ are the ones most relevant to our study of the morphism $\bA$, thus in the present paper only these two varieties are considered.
\end{remark}

We have the obvious inclusions
\begin{equation}
U_{\Res}\subset\Gor(T)\subset U\subset V,\label{setincl}
\end{equation}
where $U_{\Res}\subset \CC[y_1,\dots,y_n]_{n(d-1)}$ was defined in Section \ref{morphism}. To better understand the relationship between $U_{\Res}$, $\Gor(T)$, $U$ and $V$, we will now introduce a certain closed subset of $U$. 

Cover $U$ by open subsets ${\rm U}_{\alpha}$, each of which is given by the condition of the nonvanishing of a particular $(K-n)$-minor of the catalecticant matrix $D(F)$. In what follows, on each ${\rm U}_{\alpha}$ we will define a regular function $R_{\alpha}$. Let, for instance, ${\rm U}_{\alpha_{{}_0}}$ be the subset of $U$ described by the nonvanishing of the principal $(K-n)$-minor of $D(F)$. For $F\in {\rm U}_{\alpha_{{}_0}}$ we will now find a canonical basis of the solution set ${\mathcal S}(F)$ of the homogeneous system
$D(F)\gamma=0$,
where $\gamma$ is a column-vector in $\CC^K$. Since $\rank D(F)=K-n$, one has $\dim_{\CC}{\mathcal S}(F)=n$. Split $D(F)$ into blocks as follows:
$$
D(F)=\left(
\begin{array}{c}
\boxed{A(F)} \quad \boxed{B(F)}\\
\vspace{-0.3cm}\\
\boxed{\hspace{0.7cm}C(F)\hspace{0.7cm}}
\end{array}
\right),
$$
where $A(F)$ has size $(K-n)\times(K-n)$ (recall that $\det A(F)\ne 0$), $B(F)$ has size $(K-n)\times n$, and $C(F)$ has size $(L-K+n)\times K$. We also split the column-vector $\gamma$ as $\gamma=\left(\begin{array}{c}\gamma'\\\gamma''\end{array}\right)$, where $\gamma$ is in $\CC^{K-n}$ and $\gamma''$ is in $\CC^n$. Then ${\mathcal S}(F)$ is given by the  condition
$
\gamma'=-A(F)^{-1}B(F)\gamma''.
$
Therefore, the vectors
$$
\gamma_j(F):=\left(\begin{array}{c}-A(F)^{-1}B(F){\mathbf e}_j\\{\mathbf e}_j\end{array}\right),\quad j=1,\dots,n,
$$
form a basis of ${\mathcal S}(F)$ for every $F\in {\rm U}_{\alpha_{{}_0}}$, where ${\mathbf e}_j$ is the $j$th standard basis vector in $\CC^n$. 

Clearly, the components $\gamma_j^1,\dots,\gamma_j^K$ of $\gamma_j$ are regular functions on ${\rm U}_{\alpha_{{}_0}}$ for each $j$, and we define
$r_{j,\alpha_{{}_0}}:=\sum_{i=1}^K \gamma_j^i\,{\mathtt m}_i$, $j=1,\dots,n$,
where, as before, $\{{\mathtt m}_1,\dots,{\mathtt m}_K\}$ is the standard monomial basis in $\CC[x_1,\dots,x_n]_d$. Then the $d$-forms $r_{1,{\alpha_{{}_0}}}(F),\dots,r_{n,{\alpha_{{}_0}}}(F)$ constitute a basis of the intersection $F^{\perp}\cap \CC[x_1,\dots,x_n]_d$ for every $F\in {\rm U}_{\alpha_{{}_0}}$.  Set
$
R_{\alpha_{{}_0}}:=\Res(r_{1,\alpha_{{}_0}},\dots,r_{n,\alpha_{{}_0}}).
$
Clearly, $R_{\alpha_{{}_0}}$ is a regular function on ${\rm U}_{\alpha_{{}_0}}$, and we define $Z_{\alpha_{{}_0}}$ to be its zero locus.

Arguing as above for every ${\rm U}_{\alpha}$, we introduce a regular function $R_{\alpha}$ on ${\rm U}_{\alpha}$ and its zero locus $Z_{\alpha}$. Notice that if for some $\alpha$, $\alpha'$ the intersection ${\rm U}_{\alpha,\alpha'}:={\rm U}_{\alpha}\cap {\rm U}_{\alpha'}$ is nonempty, then $Z_{\alpha}\cap {\rm U}_{\alpha,\alpha'} =Z_{\alpha'}\cap {\rm U}_{\alpha,\alpha'}$. Thus, the loci $Z_{\alpha}$ glue together into a closed subset $Z$ of $U$. If $U'$ is an irreducible component of $U$, then the intersection $Z\cap U'$ is either a hypersurface in $U'$, or all of $U'$, or empty. Notice also that $Z$ is $\GL_n$-invariant, which follows from the general formula
$$
(CF)^{\perp}\cap \CC[x_1,\dots,x_n]_j=C^{-t}\,(F^{\perp}\cap \CC[x_1,\dots,x_n]_j),\quad j=0,\dots,n(d-1),
$$
for all $C\in\GL_n$, $F\in \CC[y_1,\dots,y_n]_{n(d-1)}$.

We will now establish:

\begin{proposition}\label{onto1} One has $U_{\Res}=\Gor(T)\setminus Z=U\setminus Z=V\setminus\overline{Z}$.
\end{proposition}

\begin{proof} It is clear that $U_{\Res}=U\setminus Z$, thus inclusions (\ref{setincl}) imply $U_{\Res}=\Gor(T)\setminus Z=U\setminus Z$. Further, to see that $U\setminus Z=V\setminus\overline{Z}$, we need to prove that $V\setminus U\subset\overline{Z}$. As shown in the proof of \cite[Lemma 3.5]{IK}, in every neighborhood of every form $F\in V\setminus U$ there exists $\hat F\in U$ such that all elements of $\hat F^{\perp} \cap \CC[x_1,\dots,x_n]_d $ have a common zero away from the origin. Thus, $F\in\overline{Z}$ as required. \end{proof}

Next, by Proposition \ref{P:imagehat}, the morphism $\hat\bA: \Gr(n, \CC[x_1, \ldots, x_n]_{d})_{\Res}\to \PP(U_{\Res})$ is surjective. In fact, by \cite[Proposition 2.13]{AI2}, the map $\hat\bA$ is an isomorphism, therefore we have
$$
\dim_\CC \PP(U_{\Res})=\dim_\CC \Gr(n, \CC[x_1, \ldots, x_n]_{d})=Kn-n^2,
$$
which implies
\begin{equation}
\dim_\CC U_{\Res}=Kn-n^2+1=:N.\label{dim1}
\end{equation}
As $U_{\Res}$ is irreducible, we obtain the following result:

\begin{theorem}\label{main8} There exist irreducible components $\Gor(T)^{\circ}$, $U^{\circ}$, $V^{\circ}$ of the varieties $\Gor(T)$, $U$, $V$, respectively, such that
$
U_{\Res}=\Gor(T)^{\circ}\setminus Z=U^{\circ}\setminus Z=V^{\circ}\setminus\overline{Z},
$
with
$
\dim_\CC\Gor(T)^{\circ}=\dim_\CC U^{\circ}=\dim_\CC V^{\circ}=N,
$
where $N$ is defined in {\rm (\ref{dim1})}.
\end{theorem}

As by Proposition \ref{P:image} we have $\im(\bA)=U_{\Res}$, Theorem \ref{P:image} yields a description of the image of the morphism $\bA$ in terms of $\Gor(T)$, $U$, $V$ and $Z$.

\begin{remark}\label{compar1} Theorem 4.17 of \cite{IK} shows that $\Gor(T)$ has an irreducible component containing $U_{\Res}$ as a dense subset and the dimension of this component is equal to $N$. The proof given in \cite{IK} does not explicitly utilize the morphism $\bA$ and is somewhat brief overall. Also, Theorem 4.19 of \cite{IK} (cf.~Corollary 4.18 therein) yields that $U_{\Res}$ is dense in an irreducible component of $V$ in the following cases: (i) $n=3$, $d\ge 3$, (ii) $n=4$, $d=2,3$, (iii) $n=5$, $d=2$. In comparison with these results, Theorem \ref{main8} stated above is more precise because:
\begin{itemize}

\item it treats both $\Gor(T)$ and $V$ simultaneously for all $n,d$;

\item it shows that $U_{\Res}$ is in fact open (not just dense) in an irreducible component of each of $\Gor(T)$ and $V$ and explicitly describes the closed complement to $U_{\Res}$ in terms of the subset $Z$;

\item its proof gives a complete argument for the formula for $\dim_\CC U_{\Res}$. 

\end{itemize} 

\end{remark}

\begin{remark}\label{complement}
Describing the complement to $\im(\bA)=U_{\Res}$ in $V^{\circ}$ is of particular importance for settling Conjecture \ref{C:main}. Theorem \ref{main8} offers a description in terms of the set $Z$, but, ideally, one would like to show that there exists an $\SL_n$-invariant form on $\CC[y_1,\dots,y_n]_{n(d-1)}$ whose zero locus intersects $V^{\circ}$ in $V^{\circ}\setminus \im(\bA)$. This indeed holds for $n=2$, in which case $V^{\circ}=V=\CC[y_1,y_2]_{2(d-1)}$ and $\CC[y_1,y_2]_{2(d-1)}\setminus \im(\bA)$ is the zero locus of the catalecticant (see \cite[Proposition 4.3]{AI2}). The above fact was instrumental for establishing Conjecture \ref{C:main} in the binary case in \cite{AI2}. In the next section we will show that an analogous statement is also valid for $n=3$, $d=2$. Notice that, by \cite{EI}, the conjecture holds in this situation as well.
\end{remark}

\section{The case $n=3$, $d=2$}\label{aronhold}
\setcounter{equation}{0}

In this section we set $n=3$, $d=2$. Notice that the associated form of any element of $(\CC[x_1,x_2,x_3]_2^{\oplus 3})_{\Res}$ is a ternary cubic and that $V^{\circ}=V=\CC[y_1,y_2,y_3]_{3}$. Let $S$ be the degree four Aronhold invariant. An explicit formula for $S$ can be found, for example, in \cite[p.~250]{DK}. Namely, for a ternary cubic
$$
\begin{array}{l}
c(y_1,y_2,y_3)=ay_1^3+by_2^3+cy_3^3+3dy_1^2y_2+3ey_1^2y_3+3fy_1y_2^2+\\
\vspace{-0.1cm}\\
\hspace{6cm}3gy_2^2y_3+3hy_1y_3^2+3iy_2y_3^2+6jy_1y_2y_3
\end{array}
$$
one has
\begin{equation}
\begin{array}{l}
S(c)=abcj-bcde-cafg-abhi-j(agi+bhe+cdf)+\\
\vspace{-0.3cm}\\
\hspace{1.5cm}afi^2+ahg^2+bdh^2+bie^2+cgd^2+cef^2-j^4+\\
\vspace{-0.3cm}\\
\hspace{1.5cm}2j^2(fh+id+eg)-3j(dgh+efi)-f^2h^2-i^2d^2-\\
\vspace{-0.3cm}\\
\hspace{1.5cm}e^2g^2+ideg+egfh+fhid.
\end{array}\label{invs}
\end{equation}

We will now state the result of this section, which for $n=3$, $d=2$ provides a more explicit description of the complement $\CC[y_1,y_2,y_3]_3\setminus\im(\bA)$ than the one given by Theorem \ref{main8}.

\begin{proposition}\label{psind3} One has $\CC[y_1,y_2,y_3]_3\setminus\im(\bA)=\{S=0\}$.
\end{proposition}

\begin{proof} We utilize canonical forms of ternary cubics. Namely, every nonzero ternary cubic is linearly equivalent to one of the following:
$$
\begin{array}{ll}
c_{1,t}:=y_1^3+y_2^3+y_3^3+ty_1y_2y_3, & t^3\ne -27,\\
\vspace{-0.3cm}\\
c_2:=y_1^3+y_2^2y_3 & (\hbox{cuspidal cubic}),\\
\vspace{-0.3cm}\\
c_3:=y_1^3+y_1^2y_3+y_2^2y_3 & (\hbox{nodal cubic}),\\
\vspace{-0.3cm}\\
c_4:=y_1^2y_3+y_2y_3^2,\\
\vspace{-0.3cm}\\
c_5:=y_1^3+y_1y_2y_3,\\
\vspace{-0.3cm}\\
c_6:=y_1y_2y_3,\\
\vspace{-0.3cm}\\
c_7:=y_1^2y_2+y_1y_2^2,\\
\vspace{-0.3cm}\\
c_8:=y_1^2y_2,\\
\vspace{-0.3cm}\\
c_9:=y_1^3
\end{array}
$$
(see, e.g., \cite[p.~44]{K}). Using formula (\ref{invs}) it is now easy to deduce
$$
\{S=0\}=\{0\}\cup O(c_{1,0})\cup O(c_2)\cup O(c_4)\cup O(c_7)\cup O(c_8)\cup O(c_9),
$$
where for a ternary cubic $c$ we denote by $O(c)$ its $\GL_3$-orbit. In particular, we have $\{S=0\}=\overline{O(c_{1,0})}$, which is the closure of the locus of ternary forms representable as the sum of the cubes of three linear forms (cf. \cite[Theorems 2.1, 2.2]{Ba} and  \cite[Proposition 5.13.2]{DK}).

To see that $\im(\bA)$ does not intersect the zero locus of $S$, we find the degree two component of the annihilator of each of the cubics $c_{1,0}, c_2, c_4, c_7, c_8, c_9$:
$$
\begin{array}{l}
c_{1,0}^{\perp}\cap\CC[x_1,x_2,x_3]_2=\langle x_1x_2, x_1x_3, x_2x_3\rangle,\\
\vspace{-0.3cm}\\
c_2^{\perp}\cap\CC[x_1,x_2,x_3]_2=\langle x_1x_2,x_1x_3,x_3^2\rangle,\\
\vspace{-0.3cm}\\
c_4^{\perp}\cap\CC[x_1,x_2,x_3]_2=\langle x_1^2-x_2x_3,x_1x_2,x_2^2\rangle,\\
\vspace{-0.3cm}\\
c_7^{\perp}\cap\CC[x_1,x_2,x_3]_2=\langle x_1^2+x_2^2-x_1x_2,x_1x_3,x_2x_3,x_3^2\rangle,\\
\vspace{-0.3cm}\\
c_8^{\perp}\cap\CC[x_1,x_2,x_3]_2=\langle x_1x_3,x_2^2,x_2x_3,x_3^2\rangle,\\
\vspace{-0.3cm}\\
c_9^{\perp}\cap\CC[x_1,x_2,x_3]_2=\langle x_1x_2,x_1x_3,x_2^2,x_2x_3,x_3^2\rangle.
\end{array}
$$  
We thus see that for the cubics $c_7,c_8$, $c_9$ the corresponding annihilator components have dimension greater than 3 and that in the remaining situations they have zeroes away from the origin. It then follows that
$$
\im(\bA)\subset \CC[y_1,y_2,y_3]_{3}\setminus\{S=0\}.
$$

In order to show that $\bA$ maps $(\CC[x_1,x_2,x_3]_2^{\oplus 3})_{\Res}$ onto $\CC[y_1,y_2,y_3]_{3}\setminus\{S=0\}$, we need to prove that each of the cubics $c_{1,t}$, $c_3$, $c_5$, $c_6$ lies in $\im(\bA)$, where $t\ne 0$,\linebreak $t^3\ne 216$ (notice that $c_{1,0}$ and $c_{1,\tau}$ with $\tau^3=216$ are linearly equivalent---see, e.g., \cite[p.~603]{AIK}). First of all, $c_{1,t}$, with $t\ne 0$, $t^3\ne 216$,  is proportional to the associated form of the nondegenerate cubic $c_{1,-18/t}$ and $c_6$ to the associated form of the nondegenerate cubic $c_{1,0}$ (see, e.g., \cite[Section 2.2]{AIK}). Next, we calculate the degree two component of the annihilator of each of the cubics $c_3, c_5$:
$$
\begin{array}{l}
c_3^{\perp}\cap\CC[x_1,x_2,x_3]_2=\langle x_1^2-x_2^2-3x_1x_3,x_1x_2,x_3^2\rangle,\\
\vspace{-0.3cm}\\
c_5^{\perp}\cap\CC[x_1,x_2,x_3]_2=\langle x_1^2-6x_2x_3,x_2^2,x_3^2\rangle.
\end{array}
$$
This shows that $c_3,c_5$ lie in $U_{\Res}$ hence in $\im(\bA)$. 

The proof is now complete. \end{proof}

\end{document}